\documentclass[11pt]{article}

\usepackage{amsmath, amssymb, amsthm} 
\usepackage{hyperref}

\date{\today}
\allowdisplaybreaks
\expandafter\let\expandafter\oldproof\csname\string\proof\endcsname
\let\oldendproof\endproof
\renewenvironment{proof}[1][\proofname]{%
	\oldproof[\bf #1]%
}{\oldendproof}

\newcommand{\ignore}[1]{}

\theoremstyle{plain}
\newtheorem{theorem}{Theorem}
\newtheorem{lemma}{Lemma}[section]

\newtheorem{corollary}[lemma]{Corollary}
\newtheorem{conjecture}[theorem]{Conjecture}

\newtheorem{definition}[lemma]{Definition}

\newtheorem{fact}[lemma]{Fact}

\newcommand{\tind}{t_{\mathrm{ind}}}
\newcommand{\tinj}{t_{\mathrm{inj}}}
\newcommand{\aut}{\mathrm{aut}}

\DeclareMathOperator{\poly}{poly}

\def\bar{\overline}

\def\({\left(}
\def\){\right)}
\def\<{\langle}
\def\>{\rangle}
\let\:=\colon
\let\epsilon\varepsilon

\def\cF{\mathcal F}

\def\PP{\mathbb{P}}

\def\t{{\rm t}}

\RequirePackage[normalem]{ulem} 
\RequirePackage{color}\definecolor{RED}{rgb}{1,0,0}\definecolor{BLUE}{rgb}{0,0,1} 

\title{Testing linear inequalities of subgraph statistics}

\author{Lior Gishboliner \thanks{School of Mathematical Sciences, Tel Aviv University, Tel Aviv, 69978, Israel. Email: liorgis1@mail.tau.ac.il.}
\and Asaf Shapira \thanks{
School of Mathematical Sciences, Tel Aviv University, Tel Aviv 69978, Israel.
Email: asafico$@$tau.ac.il. Supported in part by ISF Grant 1028/16 and ERC Starting Grant 633509.}
\and Henrique Stagni \thanks{Departamento de Ciencia da Computacao,
Instituto de Matematica e Estatistica, Universidade de Sao Paulo, Brazil. Email: stagni@gmail.com.}
}

\begin{document}

\date{}

\maketitle
\begin{abstract}

Property testers are fast randomized algorithms whose task is to distinguish between inputs satisfying some predetermined property ${\cal P}$
and those that are far from satisfying it.
Since these algorithms operate by inspecting a small randomly selected portion of the input,
the most natural property one would like to be able to test is whether the input does not
contain certain forbidden small substructures. In the setting of graphs, such a result
was obtained by Alon et al., who proved that for any finite family of graphs ${\cal F}$,
the property of being induced ${\cal F}$-free (i.e. not containing an induced copy of any $F \in {\cal F}$) is testable.

It is natural to ask if one can go one step further and prove that more elaborate properties
involving induced subgraphs are also testable. One such generalization of the result of Alon et al. was formulated by Goldreich and Shinkar
who conjectured that for any finite family of graphs ${\cal F}$, and any linear inequality involving the densities of the graphs $F \in {\cal F}$ in the input graph,
the property of satisfying this inequality can be tested in a certain restricted model
of graph property testing. Our main result in this paper disproves this conjecture in the following
strong form: some properties of this type are not testable even
in the classical (i.e. unrestricted) model of graph property testing.

The proof deviates significantly from prior non-testability results in this area. The main idea is to use a linear inequality
relating induced subgraph densities in order to encode the property of being a quasirandom graph.

\end{abstract}

%


\section{Introduction}
Property testers are fast randomized algorithms that distinguish between objects satisfying a certain property and objects that are ``far" from the property. The systematic study of such problems originates in the seminal papers of Rubinfeld and Sudan \cite{RS} and Goldreich, Goldwasser and Ron \cite{GGR}, and has since become a very active area of research with strong ties to other areas; most notably discrete mathematics, and extremal combinatorics in particular. We refer the reader to the book of Goldreich \cite{Goldreich} for more background and references on the subject.

In this paper we study property testing of graph properties in the {\em dense graph model}. In this model, a graph is given as an $n \times n$ adjacency matrix. An $n$-vertex graph $G$ is said to be {\em $\varepsilon$-far} from a graph property $\Pi$, if one has to change at least $\varepsilon n^2$ entries in the adjacency matrix of $G$ in order to turn it into a graph satisfying $\Pi$. A {\em tester} for $\Pi$ is a (randomized) algorithm that, given a proximity parameter $\varepsilon > 0$ and a graph $G$, accepts if $G$ satisfies $\Pi$ and rejects if $G$ is $\varepsilon$-far from $\Pi$, with success probability at least $\frac{2}{3}$ in both cases. The tester is given oracle access to the adjacency matrix of the input, to which it makes queries. 
In this paper we focus on so-called {\em canonical} testers (see \cite{GT}). A tester is called canonical if it works as follows: given a proximity parameter $\varepsilon > 0$ and an $n$-vertex input graph $G$, the tester samples a set of $s(\varepsilon,n)$ vertices of $G$ uniformly at random, querying all pairs among these vertices, and makes its decision solely based on (the isomorphism class of) the subgraph induced by the sample. The function $s(\varepsilon,n)$ is called the {\em sample complexity} of the tester. Note that the number of edge-queries made by such a tester is $\binom{s(\varepsilon,n)}{2}$. The reason for restricting our attention to canonical testers is an important result of Goldreich and Trevisan \cite[Theorem 2]{GT} (see also \cite{GT_errata}), who showed that any tester can be transformed into a canonical one with only minor loss in efficiency. More precisely, they showed that if a property $\Pi$ has a tester with {\em query-complexity} $q(\varepsilon,n)$ --- i.e., if this tester makes at most $q(\varepsilon,n)$ edge-queries when invoked with proximity parameter $\varepsilon$ and with an $n$-vertex input graph --- then $\Pi$ also has a canonical tester with sample complexity $s(\varepsilon,n) = O(q(\varepsilon,n))$. 


A property $\Pi$ is called {\em testable} if it has a tester whose sample complexity is bounded by a function of $\varepsilon$ alone, that is, it is independent of the size of the input. 
A tester is {\em size-oblivious} if it does not know $n$; that is, if its operation depends only on the proximity parameter $\varepsilon$ (and not on the size of the input). The aforementioned transformation of Goldreich and Trevisan \cite{GT}, which turns arbitrary testers into canonical ones, preserves the property of being size-oblivious.
	
In this paper we study a special kind of testers, called {\em proximity oblivious testers}, which were first introduced and studied by Goldreich and Ron \cite{GR}, and are defined as follows.
\begin{definition}\label{def:POT}
A {\em proximity oblivious tester} (POT) for a graph property $\Pi$ is an algorithm which makes a {\em constant} (i.e. independent of $n$ {\em and} $\varepsilon$) number of queries to the input and satisfies the following. There is a constant $c \in (0,1]$ and a function
$f: (0,1] \rightarrow (0,1]$ such that:
\begin{enumerate}
		\item If the input graph satisfies $\Pi$ then the tester accepts with probability at least $c$.
		\item If the input graph is $\varepsilon$-far from $\Pi$ then the tester accepts with probability at most $c - f(\varepsilon)$.
	\end{enumerate}
	\end{definition}
	\noindent
	Observe that a POT for $\Pi$ can be used to obtain a standard tester for $\Pi$, by invoking the POT $T = \Theta(1/f(\varepsilon)^2)$ times and accepting if and only if the POT accepted in at least
	$(c - \frac{f(\varepsilon)}{2})T$ of \nolinebreak the \nolinebreak tests.
	
Goldreich and Ron \cite{GR} studied {\em one-sided-error POTs}, namely POTs that accept every input which satisfies the property with probability $1$ (this corresponds to having $c = 1$ in Definition \ref{def:POT}). Later, Goldreich and Shinkar \cite{GS} studied general (two-sided-error) POTs in several settings, including those of general boolean functions, dense graphs and bounded degree graphs. For the dense graph model, they designed a POT for the property of being $\alpha n$-regular (for a given $\alpha \in (0,1)$), as well as for several related properties. They moreover considered certain subgraph density properties, defined as follows. Given graphs $H,G$, the {\em density} of $H$ in $G$, denoted by $p(H,G)$, is the fraction of induced subgraphs of $G$ of order $|V(H)|$ which are isomorphic to $H$. We will need the following basic property of this subgraph \nolinebreak density \nolinebreak function:
\begin{fact}\label{fact:density_averaging}
	For every pair of graphs $F,G$ and $h \geq |V(F)|$, it holds that $p(F,G) = \sum_{H}{p(F,H) \cdot p(H,G)}$, where the sum is over all $h$-vertex graphs $H$.  
\end{fact}

 Given an integer $h \geq 2$, a rational number $b$ and rational numbers $w_H \geq 0$, where $H$ runs over all $h$-vertex graphs, the property $\Pi_{h,w,b}$ is defined as the property of all graphs $G$ satisfying
$$
\sum_{H}{w_H \cdot p(H,G)} \leq b\;.
$$ 
Throughout this paper, a tuple $(h,w,b)$ will always consist of an integer $h \geq 2$, a rational number $b$, and a function $w \colon\{H:v(H)=h\}\to \mathbb{Q}_{\geq 0}$ from the set of all $h$-vertex graphs to the nonnegative rationals. The value assigned by $w$ to a graph $H$ is denoted by $w_H$.
	
Since property testing algorithms, and POTs in particular, work by inspecting the subgraph induced by a small sample of vertices,
it is natural to ask if the property of not containing an induced copy of a fixed graph $H$ is a testable property.
Such a result was obtained by Alon, Fischer, Krivelevich and Szegedy \cite{AFKS} who proved that in fact for every finite family
of graphs ${\cal F}$, the property of being induced ${\cal F}$-free (i.e. not containing an induced copy of $F$ for every $F \in {\cal F}$)
is testable. It is easy to see that the family of properties $\Pi_{h,w,b}$ forms a strict generalization of the family of properties of being induced ${\cal F}$-free, since the former can encode the latter. 
Indeed, if all graphs in $\mathcal{F}$ have the same size $h$ then simply set $b = 0$, $w_H = 1$ for each $H \in \mathcal{F}$, and $w_H = 0$ for each $h$-vertex graph $H$ which is not in $\mathcal{F}$. 
If graphs in $\mathcal{F}$ have varying sizes, then take advantage of Fact \ref{fact:density_averaging}.

It is natural to ask whether $\Pi_{h,w,b}$ has a POT for every $(h,w,b)$. Such a conjecture has been indeed raised by Goldreich and Shinkar in \cite{GS}.

\begin{conjecture}[\protect{\cite[Open Problem 3.11]{GS}}]\label{que:testable}
	Every property $\Pi_{h,w,b}$ has a POT.
\end{conjecture}

Our main result, Theorem \ref{thm:counterexample}, disproves the above conjecture in a strong sense, by showing that there are properties $\Pi_{h,w,b}$ that are not
testable at all (let alone testable using a POT). 
In fact, we present a property $\Pi_{h,w,b}$ which cannot even be testable with query-complexity $n^{0.01}$, even if the approximation parameter is constant (say $0.1$).  
For a graph $H$, denote by $\bar{H}$ the complement of $H$. 

\begin{theorem}\label{thm:counterexample}
	Let $K_4$ denote the complete graph on $4$ vertices, $D_4$ the diamond graph (i.e. $K_4$ minus an edge),
	$P_3$ the graph on $4$ vertices containing a path on $3$ vertices and an isolated vertex, $C_4$ the $4$-cycle, $P_4$
	the path on $4$ vertices, and $K_{1,3}$ the star on $4$ vertices. 
	Set $b = 5/16$, and let $w_H$ be the following weight function assigning a non-negative weight to each graph on $4$ vertices. 
	\begin{center}
		\def\arraystretch{1.5}
		\begin{tabular}{| r || c | c | c | c | c | c | c | c | c | c | c| }
			\hline
			$H:$ & $K_4$ & $\bar{K_4}$ & $D_4$ & $\bar{D_4}$ &
			$P_3$ & $\bar{P_3}$ & $C_4$ & $\bar{C_4}$ &
			$K_{1,3}$ & $\bar{K_{1,3}}$ & $P_4$\\
			\hline
			$w_H:$ & $1$ & $\frac12$ & $\frac{5}{12}$ & $\frac{5}{12}$ &
			$\frac13$ & $\frac16$ & $\frac12$ & $\frac13$ &
			$\frac14$ & $\frac14$ & $\frac14$\\
			\hline
		\end{tabular}
	\end{center}
	Define the property 
	\begin{equation}\label{eqP}
	\Pi_{h,w,b}=\left\{G:\sum_{H:|V(H)|=4} w_H \cdot p(H,G)\leq \frac{5}{16}\right\}
	\end{equation}
	Then every canonical $0.1$-tester for $\Pi_{h,w,b}$ has sample complexity $s(n) > n^{0.01}$. 
\end{theorem}

By combining Theorem \ref{thm:counterexample} with the aforementioned result of \cite{GT}, we see that every $0.1$-tester for $\Pi_{h,w,b}$ must make $\Omega(n^{0.01})$ edge-queries when operating on $n$-vertex input graphs.

The bound of $n^{0.01}$ appearing in Theorem \ref{thm:counterexample} is not sharp. We believe that it would be interesting to determine the optimal sample complexity needed to test the property $\Pi = \Pi_{h,w,b}$ defined in \eqref{eqP}. At the moment, we cannot even show that a sample of size $o(n)$ suffices for testing $\Pi$. 
More generally, is it true that {\em every} property of the form  $\Pi_{h,w,b}$ can be tested with sample complexity $o(n)$?
We leave this as an open problem. At the end of Section \ref{sec:counter}, we will explain how we came to choose the \nolinebreak coefficients \nolinebreak in \nolinebreak \eqref{eqP}. 

The proof of Theorem \ref{thm:counterexample} appears in Section \ref{sec:counter}. 
The main idea behind the proof is to show that the property $\Pi_{h,w,b}$ defined in \eqref{eqP} encodes the property of being quasirandom with density $\frac{1}{2}$. More precisely, we show that if a graph $G$ satisfies the property given by \eqref{eqP}, then its edge density must be roughly $1/2$ and its $C_4$ density roughly $1/16$, which is known to imply that $G$ is quasirandom \nolinebreak (see \nolinebreak \cite{CGW}). 

Now, since graphs in $\Pi_{h,w,b}$ must be $o(1)$-quasirandom, a
large enough blowup of any graph $G\in \Pi_{h,w,b}$ is, say, $0.1$-far from
$\Pi_{h,w,b}$. This fact alone is already sufficient to show that
$\Pi_{h,w,b}$ has no size-oblivious 0.1-tester (with sample complexity independent of $n$), since such testers
cannot distinguish between a (large enough) graph $G$ and a blowup thereof.

Obtaining the $n^{\Omega(1)}$ lower bound of Theorem \ref{thm:counterexample} requires a more subtle argument, which goes as follows. 
First, we show that graphs satisfying $\Pi_{h,w,b}$ are quite common; precisely, we show that a $\poly(1/n)$-fraction of all $n$-vertex graphs satisfy $\Pi_{h,w,b}$ (see Lemma \ref{lem:non_empty}). We then use this fact to show that for every large enough $n$ and for every prescribed family $\mathcal{F}$ of graphs on $s := n^{c}$ vertices (for $c > 0$ small), there exists an $n$-vertex graph which satisfies $\Pi_{h,w,b}$ and has the ``correct"\footnote{When saying that a graph $G$ has the ``correct" fraction of $s$-vertex induced subgraphs belonging to (a given graph-family) $\mathcal{F}$, we mean that the fraction of such subgraphs in $G$ is approximately the same as one would expect to have in the random graph $G(n,1/2)$.} number of $s$-vertex induced subgraphs belonging to $\mathcal{F}$ (see Lemma \ref{lem:good_graph_in_Pi(h,w,b)}). 

Then, to obtain the stated bound on the sample-complexity of testing $\Pi_{h,w,b}$, we argue as follows. Arguing by contradiction, we suppose that there is a canonical tester $\mathcal{T}$ for $\Pi_{h,w,b}$ having sample complexity $s(n)$, and let $\mathcal{F}(n)$ be the set of ``rejections graphs" of $\mathcal{T}$ (so each graph in $\mathcal{F}(n)$ has $s(n)$ vertices). Then, setting $m := n^{c'}$ (for a suitable small $c' > 0$), we fix an $m$-vertex graph $G$ which 
has the ``correct" fraction of induced subgraphs belonging to $\mathcal{F}(n)$. We then let $\Gamma$ be the $n/m$-blowup of $G$, and let $V_1,\dots,V_{m}$ be the parts of this blowup, corresponding to the vertices of $G$. 
Now, every pair of parts $V_i,V_j$ forms either a complete or an empty bipartite graph in $\Gamma$, which means that $\Gamma$ cannot be $\frac{1}{m}$-quasirandom. It follows that in order to turn $\Gamma$ into a quasirandom graph, one must make many changes in all bipartite graphs $(V_i,V_j)$. Hence, $\Gamma$ is $\Omega(1)$-far from being quasirandom, and hence also $\Omega(1)$-far from $\Pi_{h,w,b}$. 

	Finally, we fix an $n$-vertex graph $\Gamma^*$ which satisfies $\Pi_{h,w,b}$ and also has the ``correct" fraction of induced subgraphs belonging to $\mathcal{F}(n)$.  
We then argue that as $G$ and $\Gamma^*$ have essentially the same fraction of induced subgraphs belonging to $\mathcal{F}(n)$, and as $\Gamma$ is a blowup of $G$, the tester $\mathcal{T}$ will very likely decide in the same manner on inputs $\Gamma$ and $\Gamma^*$. But as $\Gamma^*$ satisfies $\Pi_{h,w,b}$ and $\Gamma$ is far from it, $\mathcal{T}$ is not a valid tester for $\Pi_{h,w,b}$, giving the desired contradiction.



To state our second main result, we first need to introduce the following important definition. 

\begin{definition}\label{def:removal}
	A tuple $(h,w,b)$ has the {\em removal property} if there is a function $f : (0,1] \rightarrow (0,1]$ such that for every $\varepsilon \in (0,1)$ and for every graph $G$, if $G$ is $\varepsilon$-far from $\Pi_{h,w,b}$ then
	$$
	\sum_{H}{w_H \cdot p(H,G)} \geq b + f(\varepsilon)\;.
	$$
\end{definition}
As an example, the main result of \cite{AFKS} mentioned above is equivalent to the statement that if $b=0$ then $\Pi_{h,w,b}$ has the removal property.

Goldreich and Shinkar \cite{GS} observed that if $(h,w,b)$ has the removal property then $\Pi_{h,w,b}$ admits a size-oblivious POT. Indeed, given an input graph $G$, the POT works by sampling a random induced subgraph of $G$ of order $h$, and then rejecting with probability $w_H$ if the sampled subgraph is isomorphic to $H$, for each $H$ on $h$ vertices. (Note that by multiplying by a suitable constant, we may assume that $w_H \in [0,1]$ for every $H$.) 
Observe that the probability that $G$ is rejected by the tester is precisely $\sum_{H}{w_H \cdot p(H,G)}$.
Hence, if $G$ satisfies $\Pi_{h,w,b}$ then by the definition of this property, $G$ is rejected with probability at most $b$. And if $G$ is $\varepsilon$-far from $\Pi_{h,w,b}$ then by the removal property, $G$ is rejected with probability at least $b + f(\varepsilon)$. Thus, Definition \ref{def:POT} is satisfied with $c = 1 - b$.

Our second result, Theorem \ref{thm:removal}, establishes the converse of the observation described in the previous paragraph, by showing that the removal property is {\em necessary} to having a size-oblivious POT. 
	
\begin{theorem}\label{thm:removal}
For every tuple $(h,w,b)$, if $\Pi_{h,w,b}$ has a size-oblivious POT then $(h,w,b)$ has the removal property.
\end{theorem}

Theorem \ref{thm:removal} is proved in Section \ref{sec:removal}.
From this theorem it easily follows that if one representation of a property as $\Pi_{h,w,b}$ has the removal property, then all such representations have the removal property. We state this fact in the following corollary.
	
\begin{corollary}\label{cor:removal}
Let $(h,w,b)$ and $(h',w',b')$ be tuples such that $\Pi_{h,w,b} = \Pi_{h',w',b'}$. Then $(h,w,b)$ has the removal property if and only if $(h',w',b')$ has the removal property.
\end{corollary}

\paragraph{Paper overview:} 
The proof of Theorem \ref{thm:counterexample} appears in Section \ref{sec:counter}, and the proof of Theorem \ref{thm:removal} appears in Section \ref{sec:removal}. 

\section{Proof of Theorem \ref{thm:counterexample}}\label{sec:counter}
Let $\Pi_{h,w,b}$ be as in the statement of Theorem \ref{thm:counterexample}.
As a first step towards proving the theorem, we give a different description of $\Pi_{h,w,b}$ in terms of injective densities of edges and $4$-cycles, see Lemma \ref{lem:tinj} below. First we need to introduce some notation. 
For a graph $G$, denote $$z(G):=\sum_{H:|V(H)|=4}w_H \cdot p(H,G),$$
where the weights $w_H$ are defined in the statement of Theorem \ref{thm:counterexample}. Under this notation, \linebreak
$\Pi_{h,w,b}=\{G: z(G)\leq b\}$, where $b = 5/16$.
For a pair of graphs $H$ and $G$, define
\[
\tinj(H,G) = \frac1{n^{\underline{h}}}
\lvert\{\varphi\colon V(H) \rightarrow V(G) \text{ injective s.t. }
uv\in E(H) \Rightarrow \varphi(u)\varphi(v)\in E(G)\}\rvert,
\]
and
\[
\tind(H,G) = \frac1{n^{\underline{h}}}
\lvert\{\varphi\colon V(H) \rightarrow V(G) \text{ injective s.t. }
uv\in E(H) \Leftrightarrow \varphi(u)\varphi(v)\in E(G)\}\rvert,
\]
where $n^{\underline{h}} = n\cdot(n-1)\cdot\dots\cdot(n-h+1)$.
Note that $\tind(H,G) = p(H,G) \cdot \aut(H)/h!$, where $\aut(H)$ is the number of automorphisms of $H$. 
The following lemma gives a simpler description of $\Pi_{h,w,b}$.
\begin{lemma}\label{lem:tinj}
	$\Pi_{h,w,b} = \{G:\phi(G)\leq 0\}$, where
	$\phi(G)= 2\tinj(C_4,G) - \tinj(K_2,G) + \frac{3}{8}.$
\end{lemma}
\begin{proof}
	First, note that $\tinj(K_2,G) = p(K_2,G)$.
	Now, as $C_4,D_4,K_4$ are the only $4$-vertex graphs containing $C_4$ as a subgraph, and as $C_4$ has two (labeled) supergraphs isomorphic to $D_4$ and one supergraph isomorphic to $K_4$, we have 
	\begin{align*}
	\tinj(C_4,G) &= \tind(C_4,G) + 2\tind(D_4,G) + \tind(K_4,G) \\
	&= \frac{\aut(C_4)}{4!} \cdot p(C_4,G) +
	2\cdot \frac{\aut(D_4)}{4!} \cdot p(D_4,G) +
	\frac{\aut(K_4)}{4!} \cdot p(K_4,G)\\
	&= \frac13 p(C_4,G) + \frac13 p(D_4,G) + p(K_4,G).
	\end{align*}
	Plugging the above into the definition of $\phi(G)$, we get:
	\begin{align*}
	\phi(G)&=\frac23 p(C_4,G) + \frac23 p(D_4,G) + 2 p(K_4,G) -
	p(K_2,G) + \frac38 \\
	&=\frac23 p(C_4,G) + \frac23 p(D_4,G) + 2 p(K_4,G) +
	p(\bar{K_2},G)-\frac58\\
	&=\frac23 p(C_4,G) + \frac23 p(D_4,G) + 2 p(K_4,G) +
	\sum_{H:|V(H)|=4} p(\bar{K_2},H) p(H,G)-\frac58\\
	&=\sum_{H:|V(H)|=4}2w_H \cdot p(H,G) - \frac{5}{8}.
	\end{align*}
	Here, the penultimate inequality uses Fact \ref{fact:density_averaging}.
	So we see that $\phi(G)\leq 0$ holds if and only if
	$\sum_{H:|V(H)|=4}{w_H \cdot p(H,G)} \leq 5/16$, namely if and only if $G \in \Pi_{h,w,b}$, as required. 
\end{proof}


In what follows, we will need the following well-known fact, which is closely related to the Kov\'{a}ri-S\'{o}s-Tur\'{a}n theorem \cite{KST}. For a proof of this fact, see e.g. \cite[Lemma 2.1 and Corollary 2.1]{subgraphs}.
\begin{fact}\label{C4_KST}
	Every $n$-vertex graph $G$ satisfies\footnote{Usually this inequality is stated in terms of the {\em homomorphism density}, as $\t(C_4,G) \geq \t(K_2,G)^4$. The error-term $O(\frac{1}{n})$ appearing in Fact \ref{C4_KST} accounts for the difference between the homomorphism density and the injective density, see \cite[Equation (5.21) in Section 5.2.3]{Lovasz}.}
	$\tinj(C_4,G) \geq \tinj(K_2,G)^4 - O\left(\frac{1}{n}\right)$. 
\end{fact}

We now give some background on quasirandomness. For a thorough overview of the subject, we refer the reader to \cite{Lovasz}. In what follows, we write $x = y \pm z$ to mean that $x \in [y - z, y + z]$. An $n$-vertex graph $G$ is {\em $\delta$-quasirandom} (with density $\frac{1}{2}$) if for every pair of disjoint sets $U,V \subseteq V(G)$ such that $|U|,|V| \geq \delta n$, it holds that 
$e(U,V) = \left( \frac{1}{2} \pm \delta \right)|U||V|$. 


The following seminal result\footnote{We state this result with explicit dependencies between the parameters, as this is necessary for proving our explicit lower bound (on the sample-complexity) in Theorem \ref{thm:counterexample}. The explicit dependence we state in Theorem \ref{thm:CGW} appears only implicitly in \cite{CGW}. It is very likely that this dependence is not optimal (to the best of our knowledge, the optimal dependence is not known). Having better dependence would result in a better lower bound in Theorem \ref{thm:counterexample}.} of Chung, Graham and Wilson \cite{CGW} states that quasirandomness essentially boils down to having the ``right" densities of edges and $4$-cycles. 
\begin{theorem}[\cite{CGW}]\label{thm:CGW}
	For every $\delta \in (0,1)$ there are $\gamma = \gamma(\delta) = \Omega(\delta^{12})$ and $n_0 = n_0(\delta) = O(\delta^{-12})$ such that if a graph $G$ on at least $n_0$ vertices satisfies 
	\begin{equation}\label{eq:quasirandom}
	\tinj(K_2,G) = \frac{1}{2} \pm \gamma \; \text{ and } \;
	\tinj(C_4,G) \leq \frac{1}{16} + \gamma,
	\end{equation}
	then $G$ is $\delta$-quasirandom.
\end{theorem} 
An important ingredient in the proof of Theorem \ref{thm:counterexample} is the following lemma, which shows that graphs that satisfy $\Pi_{h,w,b}$ must be quasirandom.
\begin{lemma}\label{lem:quasirandom}
	For every $n \geq 1$, every $n$-vertex graph $G \in \Pi_{h,w,b}$ is $\delta$-quasirandom with $\delta = O(n^{-1/24})$. 
\end{lemma}
\begin{proof}
	Let $n \geq 1$ and let $G$ be an $n$-vertex graph which satisfies $\Pi_{h,w,b}$. Fix the smallest $\delta \in (0,1)$ with the property that $n_0(\delta) \leq n$ and $\gamma(\delta) \geq (C/n)^{1/2}$, where $n_0(\delta)$ and $\gamma(\delta)$ are from Theorem \ref{thm:CGW}, and $C$ is some absolute constant, to be chosen later. The parameter-dependencies in Theorem \ref{thm:CGW} imply that $\delta = O(n^{-1/24})$. To prove the lemma, our goal is to show that $G$ is $\delta$-quasirandom. 
	In light of Theorem \ref{thm:CGW}, it is enough to show that $G$ satisfies \eqref{eq:quasirandom} with $\gamma = (C/n)^{1/2}$. 

By Fact \ref{C4_KST}, we have
$\tinj(C_4,G) \geq \tinj(K_2,G)^4 - C/n,$ where $C$ is a suitable absolute constant.
This implies that
\begin{equation}\label{eq:density_polynomial_inequality}
\begin{split}
2\tinj(K_2,G)^4 - \tinj(K_2,G) + \frac{3}{8} &\leq
2\tinj(C_4,G) + \frac{C}{n} - \tinj(K_2,G) + \frac{3}{8} \\ &= \phi(G) + \frac{C}{n} \leq \frac{C}{n} = \gamma^2,
\end{split} 
\end{equation}
where the second inequality follows from Lemma \ref{lem:tinj}. 

Let $f(x) := 2x^4 - x + \frac{3}{8}$ for $x \in [0,1]$. With this notation, we can rephrase \eqref{eq:density_polynomial_inequality} as $f(\tinj(K_2,G)) \leq \nolinebreak \gamma^2$.  
Note that the function
$f$ is convex, and attains its minimum at $x = 1/2$. Therefore, if we had $\tinj(K_2,G) > \frac12 + \gamma$, then we would have
\[f(\tinj(K_2,G)) = 2\tinj(K_2,G)^4 - \tinj(K_2,G) + \frac{3}{8} >
2\left( \frac{1}{2}+\gamma \right)^4 - \left( \frac{1}{2}+\gamma \right) + \frac{3}{8}
= 2\gamma^4 + 4\gamma^3 + 3\gamma^2 > \gamma^2.\]
Similarly, if we had $\tinj(K_2,G) < \frac12 - \gamma$, then we would have
\[f(\tinj(K_2,G)) = 2\tinj(K_2,G)^4 - \tinj(K_2,G) + \frac{3}{8} >
2\left( \frac{1}{2}-\gamma \right)^4 - \left( \frac{1}{2}-\gamma \right) + \frac{3}{8}
= 2\gamma^4 - 4\gamma^3 + 3\gamma^2 > \gamma^2.\]

In any case, we see that $|\tinj(K_2,G) - \frac{1}{2}| > \gamma$ would stand in contradiction to \eqref{eq:density_polynomial_inequality}.
Hence,
$\tinj(K_2,G)=\frac12 \pm \gamma$. By using again the fact that $\phi(G) \leq 0$ (see Lemma \ref{lem:tinj}), we get that 
$$\tinj(C_4,G) \leq \frac{\tinj(K_2,G)}{2} - \frac3{16} \leq \frac{1}{4} + \frac{\gamma}{2} - \frac{3}{16} < \frac{1}{16} + \gamma.$$ 
We have thus shown that \eqref{eq:quasirandom} holds (for our particular choice of $\gamma$), as required. 
%
%
%
\end{proof}

Next, we argue that for every $n$, a sizable portion of all $n$-vertex graphs satisfy $\Pi_{h,w,b}$. More precisely, we show that the probability that the random graph $G(n,1/2)$ satisfies $\Pi_{h,w,b}$ vanishes only polynomially (with $n$). 
\begin{lemma}\label{lem:non_empty}
	Let $n \geq 4$ and let $G \sim G(n,1/2)$. Then $G \in \Pi_{h,w,b}$ with probability at least $\frac{1}{2n^4}$.
\end{lemma}
\begin{proof}
	By Lemma \ref{lem:tinj}, $G \in \Pi_{h,w,b}$ if and only if $\phi(G) \leq 0$. 
	It is easy to see that $\mathbb{E}[\tinj(K_2,G)] = \frac{1}{2}$ and $\mathbb{E}[\tinj(C_4,G)] = \frac{1}{16}$. Hence, 
	$$
	\mathbb{E}[\phi(G)] = 2\mathbb{E}[\tinj(C_4,G)] - \mathbb{E}[\tinj(K_2,G)] + \frac{3}{8} = 0.
	$$
	Now, observe that for every $n$-vertex graph $G$, the densities 
	$\tinj(C_4,G)$ and $\tinj(K_2,G)$ are integer multiples of $1/n^{\underline{4}} = \frac{1}{n(n-1)(n-2)(n-3)}$. It follows that 
	$\phi(G) = 2\tinj(C_4,G) - \tinj(K_2,G) + \frac{3}{8}$ is also an integer multiple of $1/n^{\underline{4}} \geq 1/n^4$ (here we use the fact that $n^{\underline{4}}$ is divisible by $8$). Hence, if $\phi(G) > 0$ (i.e., if $G \notin \Pi_{h,w,b}$) then in fact $\phi(G) \geq 1/n^4$. 
	On the other hand, it is evident that $\phi(G) \geq -1$ for every graph $G$. By combining these two facts, we get that
	\begin{align*}
	0 &= \mathbb{E}[\phi(G)] = \mathbb{E}[\phi(G) \; | \; G \in \Pi_{h,w,b}] \cdot \mathbb{P}[G \in \Pi_{h,w,b}] + 
	\mathbb{E}[\phi(G) \; | \; G \notin \Pi_{h,w,b}] \cdot \mathbb{P}[G \notin \Pi_{h,w,b}] 
	\\ &\geq 
	-1 \cdot \mathbb{P}[G \in \Pi_{h,w,b}] + \frac{1}{n^4} \cdot \mathbb{P}[G \notin \Pi_{h,w,b}] = 
	\frac{1}{n^4} - \left( 1 + \frac{1}{n^4} \right) \cdot \mathbb{P}[G \in \Pi_{h,w,b}]. 
	\end{align*}
	It follows that $\mathbb{P}[G \in \Pi_{h,w,b}] \geq \frac{1}{n^4 + 1} \geq \frac{1}{2n^4}$, as required. 
\end{proof}
Since a polynomially-large portion of all $n$-vertex graphs satisfy $\Pi_{h,w,b}$ (by Lemma \ref{lem:non_empty}), {\em most} $n$-vertex graphs in $\Pi_{h,w,b}$ satisfy any given property which is satisfied by $G(n,1/2)$ with probability that is superpolynomially (say, exponentially) close to $1$. In particular, we can always find an $n$-vertex graph in $\Pi_{h,w,b}$ that satisfies any given property which is very likely to be satisfied by $G(n,1/2)$.  
An example of such a statement is Lemma \ref{lem:good_graph_in_Pi(h,w,b)} below. To prove this lemma, we will need the following simple version of Azuma's inequality (see, e.g.,~\cite[Theorem 2.27]{JLR}). 

\begin{lemma}\label{lem:Azuma}{\cite[Theorem 2.27]{JLR}}
	Let $X$ be a non-negative random variable, not identically $0$, which is determined by $N$ independent trials $w_1, \ldots, w_N$. Suppose that $K \in \mathbb{R}$ is such that changing the outcome of any one of the trials can change the value of $X$ by at most $K$. Then, for every $\lambda \geq 0$, 
	$$
	\mathbb{P}\left[ |X - \mathbb{E}[X]| > \lambda \right] \leq 2e^{- \frac{\lambda^2}{2 K^2 N}}.
	$$
\end{lemma}

\noindent
For a family of (unlabeled) graphs $\mathcal{F}$ and a graph $G$, define 
$
p(\mathcal{F},G) := \sum_{F \in \mathcal{F}}{p(F,G)}\;.
$
\begin{lemma}\label{lem:good_graph_in_Pi(h,w,b)} 
	There exists $n_0$ such that the following holds for every $n \geq n_0$. Let 
	$s \leq n^{0.49}$ and let $\mathcal{F}$ be a family of (unlabeled) $s$-vertex graphs. Then there is an $n$-vertex graph $G \in \Pi_{h,w,b}$ such that
	\begin{equation}\label{eq:subgraph_count_concentration}
	p(\mathcal{F},G) = \sum_{F \in \mathcal{F}}{2^{-\binom{s}2}\frac{s!}{\aut(F)}} \pm 0.1.
	\end{equation}
\end{lemma}
\begin{proof}
	Consider $G \sim G(n,1/2)$. In light of Lemma \ref{lem:non_empty}, it is enough to show that \eqref{eq:subgraph_count_concentration} holds with probability larger than $1 - \frac{1}{2n^4}$. So consider the random variable $p(\mathcal{F},G)$. It is easy to see that 
	$$
	\mathbb{E}[p(\mathcal{F},G)] = \sum_{F \in \mathcal{F}}{2^{-\binom{s}2}\frac{s!}{\aut(F)}} \; . 
	$$
	Evidently, $p(\mathcal{F},G)$ is determined by the outcome of $\binom{n}{2}$ independent trials (one per pair of vertices of $G$). Furthermore, changing the relation of a single pair of vertices (namely, changing the outcome of a single trial) can change $p(\mathcal{F},G)$ by at most $\binom{n-2}{s-2}/\binom{n}{s} = \frac{s(s-1)}{n(n-1)} \leq (s/n)^2$. Thus, we may apply Lemma \ref{lem:Azuma} to the random variable $X := p(\mathcal{F},G)$ with $N := \binom{n}{2}$ and $K := (s/n)^2$. For \nolinebreak $\lambda = 0.1$, \nolinebreak we \nolinebreak obtain 
	$$
	\mathbb{P}\left[ |X - \mathbb{E}[X]| > 0.1 \right] \leq 
	2\exp\left( {\frac{0.01}{2(s/n)^4\binom{n}{2}}} \right) \leq 
	2\exp\left( {-\frac{n^2}{100s^4}} \right) < \frac{1}{2n^4},
	$$
	where the last inequality holds for large enough $n$, as $s \leq n^{0.49}$. So we see that with probability larger than $1 - \frac{1}{2n^4}$, it holds that 
	$p(\mathcal{F},G) = \mathbb{E}[p(\mathcal{F},G)] \pm 0.1$, as required.  
\end{proof}

\noindent
We are now ready to prove Theorem \ref{thm:counterexample}. 
\begin{proof}[Proof of Theorem \ref{thm:counterexample}]
	Suppose, for the sake of contradiction, that $\Pi_{h,w,b}$ has a canonical $0.1$-tester $\mathcal{T}$ whose sample complexity $s(n)$ is at most $n^{0.01}$. 
	In what follows, we will assume that $n$ is large enough wherever needed. 	
	As $\mathcal{T}$ is canonical, for every $n \geq 1$ there is a family $\mathcal{F}=\mathcal{F}(n)$ of (rejection) graphs of order $s(n)$ such that when invoked on input graphs with $n$ vertices, $\mathcal{T}$ rejects if and only if the subgraph induced by its sample belongs to $\mathcal{F}(n)$. The fact that $\mathcal{T}$ is a valid $0.1$-tester implies that the following holds for every $n$-vertex graph $G$. 
	\begin{enumerate}
		\item $p(\mathcal{F},G) \leq \frac13$ if $G\in \Pi_{h,w,b}$;
		\item $p(\mathcal{F},G)\geq \frac23$ if $G$ is $0.1$-far from
		$\Pi_{h,w,b}$.
	\end{enumerate}
	
	By Lemma \ref{lem:quasirandom}, there is an absolute constant $c > 0$ such that for every $n$, every $n$-vertex graph $G \in \Pi_{h,w,b}$ is $\delta$-quasirandom with $\delta = cn^{-1/24}$. Choose $n$ to be large enough so that $c^{-1} n^{1/24} \geq n_0$, where $n_0$ is from Lemma \ref{lem:good_graph_in_Pi(h,w,b)}. Set $m := c^{-1}n^{1/24}$ and $\mathcal{F} := \mathcal{F}(n)$, noting that every graph in $F \in \mathcal{F}$ satisfies $|V(F)| = s(n) \leq n^{0.01} \leq m^{0.49}$. By Lemma \ref{lem:good_graph_in_Pi(h,w,b)} (applied with $m$ in place of $n$ and with $s = s(n)$),  
	there is an $m$-vertex graph $G \in \Pi_{h,w,b}$ which satisfies  \eqref{eq:subgraph_count_concentration}. For convenience, we assume that $V(G) = [m]$.
	Let $\Gamma$ be the $\frac{n}{m}$-blow-up of $G$. That is, $\Gamma$ is obtained from $G$ by replacing each vertex $i \in [m] = V(G)$ with a vertex-set $V_i$ of size $n/m$, and replacing edges (resp. non-edges) of $G$ with complete (resp. empty) bipartite graphs. Note that $|V(\Gamma)| = n$. 
	
	We claim that $\Gamma$ is $0.1$-far from $\Pi_{h,w,b}$. Indeed, fix any
	$\Gamma'\in\Pi_{h,w,b}$ with $n$ vertices. 
	By Lemma \ref{lem:quasirandom}, $\Gamma'$ is $\delta$-quasirandom for $\delta = cn^{-1/24} = 1/m$. 
 	As $|V_1| = \dots = |V_{m}| = n/m$, this $\frac{1}{m}$-quasirandomness implies that
	$e_{\Gamma'}(V_i,V_j) = (\frac12\pm \frac{1}{m}) \cdot (n/m)^2$ for every pair $1 \leq i < j \leq m$.
	But since $e_{\Gamma}(V_i,V_j)\in\{0,(n/m)^2\}$, we must change at least $(\frac12 - \frac{1}{m})(n/m)^2 \geq 0.4(n/m)^2$ edges between $V_i$ and $V_j$ for every $1 \leq i < j \leq m$, in order to turn $\Gamma$ into $\Gamma'$ (here we assume that $m \geq 10$, which holds for large enough $n$). Therefore, the distance between $\Gamma$ and $\Gamma'$ is at least
	$\binom{m}{2} \cdot 0.4 (n/m)^2 \geq 0.1n^2$. 
	This shows that $\Gamma$ is indeed $0.1$-far from $\Pi_{h,w,b}$, as required.
	

	Next, we claim that $p(\mathcal{F},\Gamma) \leq p(\mathcal{F},G) + 0.1$, where, as before, $\mathcal{F} = \mathcal{F}(n)$. To this end, set $s = s(n)$, let $S\in\binom{V(\Gamma)}{s}$ be chosen uniformly at random,
and let $\mathcal{B}$ be the event that there exists $1 \leq i \leq m$ for which $|S\cap V_i| > 1$. Note that 
$\PP(\mathcal{B}) \leq \binom{s}{2}/m = o(1) \leq 0.1$, as $s \leq n^{0.01}$ and $m = \Theta(n^{1/24})$. 
Observe that conditioned on $\mathcal{B}^c$, the probability that $\Gamma[S]$ is isomorphic to a given $s$-vertex graph $F$ is exactly $p(F,G)$.
Hence, 
\begin{equation}\label{eq:blowupbound}
\begin{split} 
p(\cF,\Gamma) = \mathbb{P}[\Gamma[S] \in \mathcal{F}] \leq
\mathbb{P}[\mathcal{B}] + 
\PP\big[ \Gamma[S]\in\cF \mid \mathcal{B}^c \big] = 
\mathbb{P}[\mathcal{B}] + p(\mathcal{F},G) \leq
0.1 + p(\mathcal{F},G),
\end{split} 
\end{equation}
as required. 

Now apply Lemma \ref{lem:good_graph_in_Pi(h,w,b)} again to obtain an $n$-vertex graph 
$\Gamma^*\in\Pi_{h,w,b}$ such that $p(\mathcal{F},\Gamma^*) = \rho \pm 0.1$, where we set
$$
\rho := \sum_{F \in \mathcal{F}}{2^{-\binom{s}2}\frac{s!}{\aut(F)}}.
$$ 
Our choice of $G$ via Lemma \ref{lem:good_graph_in_Pi(h,w,b)} implies that $p(\mathcal{F},G) = \rho \pm 0.1$ as well. We conclude that 
$p(\mathcal{F},\Gamma) \leq p(\mathcal{F},G) + 0.1 \leq \rho + 0.2 \leq p(\mathcal{F},\Gamma^*) + 0.3$. On the other hand, $p(\mathcal{F},\Gamma^*) \leq 1/3$ (as $\Gamma^*$ satisfies $\Pi_{h,w,b}$, see Item (a) above) and $p(\mathcal{F},\Gamma) \geq 2/3$ (as $\Gamma$ is $0.1$-far from $\Pi_{h,w,b}$, see Item (b) above). We have thus arrived at a contradiction, completing the proof. 
\end{proof}

A careful examination of the proof of Lemma \ref{lem:quasirandom} can reveal how we came up with the function $\phi(G)= 2\tinj(C_4,G) - \tinj(K_2,G) + \frac{3}{8}$, from which we then obtained the choice of weight function $w$ and independent coefficient $b$ appearing in the statement of Theorem \ref{thm:counterexample}. 
Evidently, our plan for proving Theorem \ref{thm:counterexample} was to find a linear inequality involving subgraph densities, which encodes the property of being quasirandom (with density $\frac{1}{2}$). 
Since quasirandomness depends only on the densities of edges and $4$-cycles (see Theorem \ref{thm:CGW}), it is natural to look for an inequality involving only these two parameters.  
Since every quasirandom graph satisfies $\tinj(C_4,G) \approx \tinj(K_2,G)^4$, it makes sense to try the following heuristic: start with a polynomial of the form $p(x) = x^4 + ax + b$, plug in $x = \tinj(K_2,G)$ and replace $x^4 = \tinj(K_2,G)^4$ with $\tinj(C_4,G)$, hoping that the resulting linear inequality $\tinj(C_4,G) + a \cdot \tinj(K_2,G) + b \leq 0$ will have the required properties. For this to work, it is necessary that the polynomial $p$ has a global minimum at $x = \frac{1}{2}$ and that $p$ equals $0$ at this point (so as to force graphs satisfying the inequality to have density $\frac{1}{2}$). Solving the constraints $p(\frac{1}{2}) = p'(\frac{1}{2}) = 0$ for $a$ and $b$, one obtains $a = -\frac{1}{2}$ and $b = \frac{3}{16}$. Multiplying the resulting $p$ by $2$, one recovers the aforementioned function $\phi(G)$.    

\section{Proof of Theorem \ref{thm:removal}}\label{sec:removal}
In this section we prove Theorem \ref{thm:removal} and Corollary \ref{cor:removal}. We will need the following auxiliary lemma.
\begin{lemma}\label{lem:exponential_error_probability}
	Suppose that a graph property $\Pi$ has a canonical size-oblivious $\varepsilon$-tester $\mathcal{T}$ with sample complexity $s = s(\varepsilon)$. 
	Then for every $n \geq s^4$ and for every $n$-vertex graph $G$ which is $\varepsilon$-far from $\Pi$, the following holds. For $U$ chosen uniformly at random from $\binom{V(G)}{s^4}$, we have
	$\mathbb{P}[ G[U] \in \Pi ] \leq e^{-\Omega(s)}$.
\end{lemma}

\begin{proof}
	We use a double-sampling trick which is implicit in \cite{GGR}. 
Let $\mathcal{A}$ be the family of all $s$-vertex graphs $A$ such that $\mathcal{T}$ accepts if the subgraph induced by its sample is isomorphic to $A$.
For a graph $G$, we say that a sequence of subsets $S_1,\dots,S_s \in \binom{V(G)}{s}$ is {\em good} if $G[S_i] \in \mathcal{A}$ for at least half of the values of $1 \leq i \leq s$; otherwise $S_1,\dots,S_s$ is {\em bad}. For a sequence of vertices $W = (x_1,\dots,x_{s^2})$, we say that $W$ is good (resp. bad) if $\{x_1,\dots,x_s\},\{x_{s+1},\dots,x_{2s}\},\dots,
\{x_{s^2-s+1},\dots,x_{s^2}\}$ is good (resp. bad).
Note that for a random $S \in \binom{V(G)}{s}$, if $G \in \Pi$ then $\mathbb{P}[G[S] \in \mathcal{A}] \geq \frac{2}{3}$, and if $G$ is $\varepsilon$-far from $\Pi$ then $\mathbb{P}[G[S] \in \mathcal{A}] \leq \frac{1}{3}$.
Using a standard Chernoff-type bound, one can show that the following holds for $S_1,\dots,S_s \in \binom{V(G)}{s}$ chosen uniformly at random and independently.
\begin{enumerate}
	\item If $G$ satisfies $\Pi$ then $S_1,\dots,S_s$ is good with probability at least $1 - e^{-Cs}$.
	\item If $G$ is $\varepsilon$-far from $\Pi$ then $S_1,\dots,S_s$ is bad with probability at least $1 - e^{-Cs}$.
\end{enumerate}
In both items above, $C > 0$ is an absolute constant.

The probability that there exists a pair $1 \leq i < j \leq s$ for which $S_i \cap S_j \neq \emptyset$ is at most $\binom{s}{2}\frac{s^2}{n} < \frac{1}{2}$, where the inequality follows from the assumption that $n \geq s^4$. It follows that with probability larger than $\frac{1}{2}$, the sets $S_1,\dots,S_s$ are pairwise-disjoint. Conditioned on the event that $S_1,\dots,S_s$ are pairwise-disjoint, the set $S_1 \cup \dots \cup S_s$ has the distribution of an element of $\binom{V(G)}{s^2}$ chosen uniformly at random. 
Thus, a random sequence of vertices $W = (x_1,\dots,x_{s^2})$ chosen {\em without repetition} from a given graph $G$ satisfies the following.
\begin{enumerate}
	\item If $G$ satisfies $\Pi$ then $W$ is good with probability at least $1 - 2e^{-Cs}$.
	\item If $G$ is $\varepsilon$-far from $\Pi$ then $W$ is bad with probability at least $1 - 2e^{-Cs}$.
\end{enumerate} 

Now let $G$ be a graph on $n \geq s^4$ vertices which is $\varepsilon$-far from $\Pi$. Consider a random pair $(U,W)$, where $U$ is chosen uniformly at random from $\binom{V(G)}{s^4}$, and $W = (x_1,\dots,x_{s^2})$ is a sequence of vertices sampled randomly without repetition from $U$. Then the marginal distribution of $W$ is that of a uniform sequence of $s^2$ vertices of $G$, sampled without repetition. Thus, by viewing $W$ as a sample from $G$ (and recalling that $G$ is $\varepsilon$-far from $\Pi$), we see that
$\mathbb{P}[W \text{ is good}] \leq 2e^{-Cs}$ (by Item 2 above). On the other hand, 
if $G[U] \in \Pi$, then, by viewing $W$ as a sample from $G[U]$, we see that 
$\mathbb{P}[W \text{ is good} \; | \; U] \geq 1 - 2e^{-Cs}$ (by Item 1 above). 
By combining these two facts, we conclude that
$$\mathbb{P}[G[U] \in \Pi] \leq 
\frac{\mathbb{P}[W\text{ is good}]}{\mathbb{P}[W\text{ is good} \; | \; G[U] \in \Pi]} 
\leq \frac{2e^{-Cs}}{1 - 2e^{-Cs}} \leq 4e^{-Cs} = e^{-\Omega(s)}.$$
\end{proof}

\begin{proof}[Proof of Theorem \ref{thm:removal}]
	Let $(h,w,b)$ be a tuple for which $\Pi_{h,w,b}$ has a size-oblivious POT. 
	As mentioned in the introduction, a POT for $\Pi_{h,w,b}$ can be used to obtain a standard tester for $\Pi_{h,w,b}$ by invoking the POT an appropriate number of times. Moreover, it is clear that if the POT is size-oblivious, then so is the resulting tester. Hence, $\Pi_{h,w,b}$ has a size-oblivious standard tester $\mathcal{T}'$ (whose query complexity is independent of the size of the input). Next, we apply to $\mathcal{T}'$ the transformation of Goldreich and Trevisan \cite{GT} to obtain a canonical tester $\mathcal{T}$ for $\Pi_{h,w,b}$. Since this transformation preserves the property of being size-oblivious, $\mathcal{T}$ is size-oblivious, and hence satisfies the condition of Lemma \ref{lem:exponential_error_probability}. 
	Denote by $s = s(\varepsilon)$ the sample complexity of $\mathcal{T}$. We may and will assume that $s$ is large enough as a function of the parameters $h$ and $b$. 
	
%

Put $z(G) := \sum_{H}{w_H \cdot p(H,G)}$.
By multiplying the inequality $\sum_{H}{w_H \cdot p(H,G)} \leq b$ by an appropriate integer, we may assume, without loss of generality, that $b$ and all weights $(w_H : H)$ are \nolinebreak integers.


Let $\varepsilon \in (0,1]$ and let $G$ be a graph which is $\varepsilon$-far from $\Pi_{h,w,b}$. Our goal is to show that $z(G) \geq b + f(\varepsilon)$, for a function $f : (0,1] \rightarrow (0,1]$ to be chosen later. Suppose first that $n < s^4$. As $G$ does not satisfy $\Pi_{h,w,b}$, we have $z(G) = \sum_{H}{w_H \cdot p(H,G)} > b$. Now, since $b$ and $(w_H : H)$ are all integers, and as $p(H,G)$ is an integer multiple of $\binom{n}{h}^{-1}$ for every $H$, we must have 
$$
z(G) \geq b + \binom{n}{h}^{-1} > b + n^{-h} > b + s^{-4h}, 
$$
implying that our assertion holds with $f(\varepsilon) = s(\varepsilon)^{-4h}$ in this case.

Suppose now that $n \geq s^4$, which is necessary in order to apply Lemma \ref{lem:exponential_error_probability}.
By Lemma \ref{lem:exponential_error_probability}, a randomly chosen
$U \in \binom{V(G)}{s^4}$ satisfies $G[U] \notin \Pi_{h,w,b}$ with probability at least $1 - e^{-\Omega(s)}$. As before, we observe that if a $k$-vertex graph $K$ does not satisfy $\Pi_{h,w,b}$, then necessarily
$$
z(K) = \sum_{H}{w_H \cdot p(H,K)} \geq b + \binom{k}{h}^{-1} > b + k^{-h}\;,
$$
as $b$ and all weights $w_H$ are integers.
Thus, if $G[U] \notin \Pi_{h,w,b}$ then
$$
z(G[U]) > b + |U|^{-h} = b + s^{-4h}\;.
$$ 
Observe (crucially) that $z(G)$ is the average of $z(G[U])$ over all
$U \in \binom{V(G)}{s^4}$. Thus, using the guarantees of Lemma \ref{lem:exponential_error_probability}, we obtain
$$
z(G) \geq (1 - e^{-\Omega(s)})(b + s^{-4h}) > b + \frac{1}{2}s^{-4h},
$$
where the last inequality holds provided that $s$ is large enough as a function of $h$ and $b$. So we may take the function $f$ in Definition \ref{def:removal} to be $f(\varepsilon) = \frac{1}{2}s(\varepsilon)^{-4h}$. This completes the proof.
\end{proof}
\begin{proof}[Proof of Corollary \ref{cor:removal}]
	We have established that $(h,w,b)$ satisfies the removal property if and only if $\Pi_{h,w,b}$ has a size-oblivious POT. The ``only if" part was explained in the introduction (see also \cite{GS}), and the ``if" part is the statement of Theorem \ref{thm:removal}. Since the existence of a tester (specifically, a size-oblivious POT) does not depend on the specific representation of a given property as $\Pi_{h,w,b}$, it is now clear that the corollary holds.
\end{proof}

\paragraph{Acknowledgments}
We would like to thank two anonymous referees for helpful comments which improved the presentation in this paper, and for encouraging us to improve the lower bound given by Theorem \ref{thm:counterexample} to a polynomial lower bound (i.e. $n^{\Omega(1)}$).

\end{document}